\documentclass[11pt]{amsart}
\providecommand{\XLeftMargin}{2.5cm}
\providecommand{\XTopMargin}{2.5cm}
\providecommand{\XRightMargin}{2.5cm}
\providecommand{\XBottomMargin}{2.5cm}

\usepackage{amssymb,latexsym,amsmath,amsthm,amscd}
\usepackage[all]{xy}
\xyoption{arc}
\providecommand{\XLeftMargin}{2cm}
\providecommand{\XTopMargin}{1.8cm}
\providecommand{\XRightMargin}{2cm}
\providecommand{\XBottomMargin}{1.8cm}

\usepackage[left=\XLeftMargin,top=\XTopMargin,right=\XRightMargin,bottom=\XBottomMargin]{geometry}
\usepackage{graphicx}
\newlength\XXXMyLength\makeatletter
\@ifclassloaded{amsart}{\def\Xadjustleft[#1]{}
}{
\def\Xadjustleft[#1]{\setlength\XXXMyLength{#1}\ifnum\numexpr\leftmargin>\numexpr\XXXMyLength\hspace{-\XXXMyLength}\else\hspace{-\leftmargin}\fi}
}
\makeatother

\UseComputerModernTips

\theoremstyle{plain}
\newtheorem{thm}{Theorem}[section]
\newtheorem*{thm*}{Theorem}
\newtheorem{lemma}[thm]{Lemma}

\newtheorem*{lem*}{Lemma}
\newtheorem{prop}[thm]{Proposition}
\newtheorem*{prop*}{Proposition}

\newtheorem*{cor*}{Corollary}
\newtheorem*{conv}{Convention}

\newtheorem*{conj*}{Conjecture}

\theoremstyle{definition}

\newtheorem*{cons*}{Construction}

\newtheorem*{df*}{Definition}
\newtheorem{nota}[thm]{Notation}
\newtheorem*{nota*}{Notation}
\newtheorem{qu}{Question}
\newtheorem*{qu*}{Question}

\newtheorem{rmk}[thm]{Remark}
\newtheorem*{rmk*}{Remark}

\newtheorem*{ex*}{Example}

\newcommand{\bB}{\mathbb{B}}
\newcommand{\bC}{\mathbb{C}}

\newcommand{\bF}{\mathbb{F}}

\newcommand{\bQ}{\mathbb{Q}}

\newcommand{\bZ}{\mathbb{Z}}

\newcommand{\cO}{\mathcal{O}}

\newcommand{\fa}{\mathfrak a}
\newcommand{\fb}{\mathfrak b}

\newcommand{\ff}{\mathfrak f}

\newcommand{\fp}{\mathfrak p}

\DeclareMathOperator{\End}{End}
\DeclareMathOperator{\Frob}{Frob}

\DeclareMathOperator{\Gal}{Gal}

\DeclareMathOperator{\Tr}{Tr}

\def\sumprime{\mathop{\sum{\raise3pt\hbox{${}'$}}}}

\makeatletter
\def\revddots{\mathinner{\mkern1mu\raise\p@
\vbox{\kern7\p@\hbox{.}}\mkern2mu
\raise4\p@\hbox{.}\mkern2mu\raise7\p@\hbox{.}\mkern1mu}}
\makeatother

\newcommand{\floor}[1]{\left\lfloor #1 \right\rfloor}

\providecommand{\abs}[1]{\left\vert #1 \right\vert}

\newcommand{\comment}[1]{}

\begin{document}

\title[$j$-Invariants of CM-Elliptic Curves Over $\bZ_p$]{On The $j$-Invariants of CM-Elliptic Curves Defined Over $\bZ_p$}
\author{Andrew Fiori}

\address{Mathematics \& Statistics
612 Campus Place N.W.
University of Calgary
2500 University Drive NW
Calgary, AB, Canada
T2N 1N4}

\begin{abstract}
We characterize the possible reductions modulo $p$ of the $j$-invariants of supersingular elliptic curves which admit complex multiplication by a (potentially non-maximal) order $\cO$ where the curve itself is defined over $\bZ_p$. In particular, we show that the collection of possible $j$-invariants as well as some aspects of the distribution depends on which primes divide the discriminant and conductor of the order $\cO$.
\end{abstract}



\maketitle

\section{Introduction}

There are several different ways of framing the results of this paper.
Our main object of study will be CM-elliptic curves over $\bZ_p$ which are supersingular at $p$.
The results we obtain will primarily be directed towards trying to address the following three questions:
\begin{enumerate}
\item When are there elliptic curves defined over $\bZ_p$, which (after base extension) admit CM by an order $\cO$ in a quadratic imaginary field $K$ in which $p$ is inert and where $p$ does not divide the conductor of $\cO$?
\item For such curves, what factors affect the possible reductions of their $j$-invariants modulo $p$ amongst the set of all supersingular $\bF_p$-rational $j$-invariants?
\item Given an $\bF_p$-rational supersingular $j$-invariant which admits CM by $\cO$, when does there exist an elliptic curve defined over $\bZ_p$, which (after base extension) admits CM by $\cO$, which reduces to it.
\end{enumerate}
Though they are not necissarily framed in this way, related questions are treated in \cite{StankewiczTwistsofShimura}, \cite{MortonGenus} and  \cite{BrillMorClassNumbers} and some of our results can naturally be viewed as generalizations to the context of non-maximal orders.
Furthermore, there are natural connections between some of our results and those presented in \cite{lauter2015arithmetic}.

We note one natural source of interest in these questions is the following observation of Ernst Kani:
\begin{prop*}
Suppose $p$ is unramified in $K$ and does not divide the conductor of $\cO$. Then every $\bF_p$ elliptic curve which admits CM by $\cO$ lifts to $\bZ_p$ (with a lifting of its CM to $\overline{\bZ}_p$) if and only if $p$ does not divide the conductor of the ring $\bZ[j(E_1),\ldots, j(E_n)]$ generated by the $j$ invariants of all elliptic curves which admit CM by $\cO$.
\end{prop*}
\begin{rmk}
This ring $\bZ[j(E_1),\ldots, j(E_n)]$ is a natural order in the ring class field of $K$ associated to $\cO$, its structure is mysterious.
\end{rmk}

The results we will describe are in contrast to what one would obtain for the same questions asked for elliptic curves over $\bZ_{p^2}$, the unramified quadratic extension of $\bZ_p$.
In particular, over $\bZ_{p^2}$, we have the following answers:
\begin{enumerate}
\item There are always CM-elliptic curves over $\bZ_{p^2}$ which admit CM by $\cO$ an order in a quadratic imaginary field $K$ in which $p$ is inert, and where $p$ does not divide the conductor.
\item From the work of Cornut-Vatsal \cite{CornutVatsal1,CornutVatsal2} and Jetchev-Kane \cite{JetchevKane} we have that the reductions of the $j$-invariants of elliptic curves with CM by $\cO$ are equidistributed among the supersingular values in $\bF_{p^2}$ (as we vary the conductors $\cO$ subject to certain congruence conditions). Moreover, for each $p$ and all but finitely many $\cO$ where $p$ is inert, the map from elliptic curves with CM by $\cO$ to supersingular $j$-invariants in $\bF_{p^2}$ is surjective.
\item By the work of Deuring \cite{Deuring} we know that given a supersingular elliptic curve $\overline{E}$ with CM by $\cO$ there always exists a lift to an elliptic curve over $\bZ_{p^2}$ with CM by $\cO$ which reduces to $\overline{E}$ (along with its CM).
\end{enumerate}

The results we obtain are motivated by computations,
some of the data from which is presented in the Appendix, 
which gave results which seemed contrary to the above. In particular if we consider only the elliptic curves which are defined over $\bZ_p$ then:
\begin{itemize}
\item 
They are not always surjective onto supersingular $\bF_p$ values as we vary $\cO$ among 
\begin{itemize} 
\item maximal orders subject to certain congruence conditions on the discriminant;
\item orders in a certain fixed $K$ subject to certain congruence conditions on the conductor;
\item orders subject to certain congruence conditions on the conductor and discriminant of $K$.
\end{itemize}
\item
The set of possible values, and hence the overall distributions depends on congruence conditions on both the discriminant of $K$ and the conductor of $\cO$.
\item
For certain congruence conditions on discriminants and conductors there are irreducible factors which always appear together, in equal numbers. So the appearance of a given factor is not independent on the appearance of another.
\end{itemize}
We should emphasize before proceeding that though perhaps unexpected in light of them, the above does not actually conflict with the aforementioned equidistribution results. 

This paper is organized as follows:
\begin{itemize}
\item In Section \ref{sec:bg} we introduce the relevant background.
\item In Section \ref{sec:res} we state and prove our results.
\item In Section \ref{sec:cq} we discuss two natural questions our work leaves open.
\item In the Appendix we discuss the computations and data on which are work is based.
\end{itemize}

\section{Background}\label{sec:bg}

In this section we will be introducing the results necessary to state and prove our theorems. 
Much of what we are saying is very well known, and can be found in many references on the theory of complex multiplication.
Some results which are perhaps less well known can be found in \cite{SchertzCM}, \cite{Deuring}, \cite{Ibukiyama}, \cite{Dorman1989global} or \cite{lauter2015arithmetic}.

\begin{conv}
Throughout this paper whenever we write $End(E)$ we shall mean the endomorphism algebra of $E$ over an algebraically closed field containing the ring of definition of $E$.
\end{conv}

We recall the following important facts:
\begin{thm}
If $E$ is an elliptic curve over a field of characteristic $0$ then either:
\begin{itemize}
\item $End(E) = \bZ$, this is the general case.
\item $End(E) = \cO$, for $\cO \subset \bQ(\sqrt{-D})$ an order in a quadratic imaginary field, this is the so-called CM-case.
\end{itemize}
\end{thm}

\begin{conv}
We shall say an elliptic curve $E$ admits CM by $\cO$ if $\End(E) \simeq \cO$. To say that $E$ admits CM does not require that we have chosen a particular isomorphism of $\cO$ with $\End(E)$.
\end{conv}

We will be interested in the CM or complex multiplication case in characteristic $0$, where we have the following classification result:
\begin{thm}
The elliptic curve $E_\tau = \bC /  (\bZ \oplus \tau\bZ)$ has $End(E) \simeq \cO$ if and only if 
\begin{enumerate}
\item $\tau \in \bQ(\sqrt{-D})$, that is $\tau$ generates a (complex) quadratic field, and
\item $\bZ + \tau\bZ \subset \bQ(\sqrt{-D})$  is a (projective) $\cO$-module.
\end{enumerate}
Moreover, for any algebraically closed field $C$ of characteristic $0$ the collection of elliptic curves which admit CM by $\cO$ is a principal homogeneous space under the action of $C\ell(\cO)$ the ideal class group of $\cO$.
\end{thm}

\begin{rmk}
Note that the collection of elliptic curves $E$ which admit CM by $\cO$ and the collection of pairs $(E,\rho:\cO\overset{\sim}\rightarrow \End(E))$ of $E$ and an isomorphism of $\cO$ with $\End(E)$ of a fixed CM-type are in bijection. In most contexts this later moduli problem is more natural. However, this moduli space never has $\bZ_p$-points. As we are primarily concerned with the field of definition of $E$ and not the field over which CM is obtained we shall be considering instead the ``moduli" of elliptic curves which admit CM by $\cO$.
\end{rmk}

\begin{thm}
If $E$ is an elliptic curve over a field of characteristic $p$ then either:
\begin{itemize}
\item $End(E) = \bZ$, this is the general case.
\item $End(E) = \cO$, for $\cO \subset \bQ(\sqrt{-D})$ an order in a quadratic imaginary field in which $p$ splits.
\item $End(E) = \bB$, for $\bB$ a maximal order in a quaternion algebra over $\bQ$ ramified only at $p$ and $\infty$. This is the so-called supersingular case.
\end{itemize}
\end{thm}

From the above we see that if ever we can reduce a CM elliptic curve $E$ at a prime inert in $K$ we will obtain a supersingular elliptic curve. In the characteristic $p$ setting it will be this case we are most interested in.

\begin{nota}
Let $m\in\bZ^+$ be square free so that $K=\bQ(\sqrt{-m})$ is the quadratic imaginary field of discriminant $D$, denote by $\cO_K$ its maximal order and $\cO = \cO_{K,\ff} = \bZ + \ff\cO_K$ an order of conductor $\ff\in \bZ$. Denote by:
\[ P_{\cO}(X) = \prod_{\fa} (X-j(\bC/\fa)). \]
where the product is over a set of representatives $\fa \triangleleft \cO$ for the class group $C\ell(\cO)$ of $\cO$.
Denote by $L$ the splitting field of $P_{\cO}(X)$ over $K$.
\end{nota}

The following facts are well known, for a reference see for example \cite{SchertzCM}.
\begin{itemize}
\item $P_{\cO}(X) \in \bZ[X]$ and is irreducible over $K$.
\item $L$ is abelian over $K$, with $\Gal(L/K) \simeq C\ell(\cO)$, the action being the natural permutation action of $C\ell(\cO)$ on the roots.
\item $L$ is galois over $\bQ$, the action of $\Gal(K/\bQ)$ on $ C\ell(\cO)$ being $g\mapsto g^{-1}$ so that $\Gal(K/\bQ)$ is a generalized dihedral group.
\item The action of complex conjugation on the ideals of $K$ agrees with the action on the set of elliptic curves which admit CM by $\cO$, which in turn agrees with the action of $\Gal(K/\bQ)$.
\item $L/K$ is ramified only at primes over $\ff$, whereas $L/\bQ$ is ramified only at primes over $D\ff$.
\end{itemize}
We shall denote by $N = \bQ(j) = \bQ[X]/(P_{\cO}(X)) \subset L$.

Based on the above we can conclude the following:
\begin{itemize}
\item If $p$ is inert in $K$ and $p$ does not divide $\ff$ (or equivalently that $\left(\frac{-D\ff^2}{p}\right) = -1$) then $p$ splits in $L/K$.
\item If $\left(\frac{-D\ff^2}{p}\right) = -1$ then $P_{\cO}(X) $ factors as a product of quadratic and linear terms over $\bZ_p$.
\end{itemize}

\begin{rmk}
The above agrees with the fact that the reductions of these elliptic curves (together with their CM actions) have models over $\bF_{p^2}$, as they are known to be supersingular.
\end{rmk}

\begin{prop}\label{prop:charpbij}
If $p$ is inert in $K$ and $E$ is an elliptic curve which admits CM by $\cO$ then the reduction of $E$ modulo $p$ is supersingular.
In particular, $\End(\overline{E}) = \bB$, where $\bB$ is a maximal order in a quaternion algebra ramified only at $p$ and infinity.
By reduction we may associate to such a curve $E$ the pair $(\End(E) \subset \End(\overline{E}))$.
This mapping gives a bijection between elliptic curves $E$, which admit CM by $\cO$, and isomorphism classes of pairs
$ (\cO \subset \bB) $
of $\cO$ with an optimal embedding into a maximal order $\bB$ as above. 

Moreover, there is a natural action of $C\ell(\cO)$ on such pairs, that is $\fa \triangleleft \cO$ takes the pair $ (\cO \subset \bB)$ to $ (\cO \subset \fa\bB\fa^{-1})$.
Under this action the set of elliptic curves which admit CM by $\cO$ is a principal homogeneous space under $C\ell(\cO)$.
In particular $\End(\overline{\fa\ast E}) = \fa\End(\overline{E})\fa^{-1}$.
\end{prop}
The original result of \cite{Dorman1989global} is corrected and generalized in \cite{lauter2015arithmetic}.

From now on we shall be working in the setting where $p$ is split in $K$ and $p$ does not divide $\ff$.
In particular we are assuming that $\left(\frac{-D\ff^2}{p}\right) = -1$.

\begin{prop}
If $P_{\cO}(X) $ has a linear factor over $\bZ_p$, the number of such linear factors is $\abs{\Gal(L/K)[2]}$ the size of the two torsion of the class group.
\end{prop}
\begin{proof}
By basic algebraic number theory we must count the size of the conjugacy class of Frobenius.
This is then a basic property to dihedral groups.
\end{proof}

\begin{rmk}
If $\abs{\Gal(L/K)[2]} = 1$ then $P_{\cO}(X) $ has a unique linear factor over $\bZ_p$.
\end{rmk}

\begin{thm}[Deuring]
If $E$ corresponds to the data $(\cO \subset \bB)$ then the reduction of $E$ modulo $p$ is defined over $\bF_p$ (rather than simply $\bF_{p^2}$) if and only if $\bB$ contains $\bZ[\sqrt{-p}]$.
\end{thm}
See \cite{Deuring}.

In \cite{Ibukiyama} Ibukiyama gives a complete classification of the maximal orders $\bB$ which contain $\bZ[\sqrt{-p}]$.

\begin{nota}
Fix $p$ and $q = 3\pmod{8}$ such that $\bB = (-p,-q)$ is the quaternion algebra ramified only at $p$ and $\infty$.
Fix $\alpha,\beta\in \bB$ such that $\alpha^2 = -p$, $\beta^2=-q$ and $\alpha\beta = -\beta\alpha$.
Choose $r \in \bZ$ such that $r^2 + p = mq$ for some $m\in \bZ$.
 
Denote:
\[   O(p,q,r,m) = \bZ + \bZ \frac{\alpha(1+\beta)}{2}+ \bZ\frac{1+\beta}{2}  + \bZ\frac{(r+\alpha)\beta}{q} \]
If $p=3\pmod{4}$ 
choose $r' \in \bZ$ such that $(r')^2 + p = 4m'q$ for some $m'\in \bZ$.
Denote:
\[   O'(p,q,r',m') = \bZ + \bZ \frac{1+\alpha}{2} + \bZ\beta + \bZ\frac{(r+\alpha)\beta}{2q}. \]
\end{nota}

\begin{thm}[Ibukiyama]
The sets $O(p,q,r,m)$ (and $O'(p,q,r',m')$) are maximal orders of $\bB$, their isomorphism classes depend only on $q$ and not on $r$ or $m$. Moreover, all pairs consisting of a maximal order in $\bB$ with an embedding of $\bZ[\sqrt{-p}]$ are of the form $O(p,q,r,m)$ (or $O'(p,q,r',m')$) with the embedding taking $\sqrt{-p} \rightarrow \pm\alpha$.

The orders $O(p,q,r,m)$ and $O'(p,q,r',m')$ are only ever isomorphic if they correspond to the $j$-invariant $1728$, equivalently if they admit an embedding of $\bZ[\frac{\sqrt{-3}+1}{2}]$.
\end{thm}
See \cite{Ibukiyama}.

\begin{rmk}
In $O(p,q,r,m)$ we may write:
\[ \alpha = 2\left(\frac{\alpha(1+\beta)}{2}\right) -q \left(\frac{(r+\alpha)\beta}{q}\right) + qr. \]
\end{rmk}

\begin{rmk}
We can count the number of isomorphism classes of $O(p,q,r,m)$ (respectively $ O'(p,q,r',m')$) by looking at the class numbers $h_{p}$ for $\bZ[\sqrt{-p}]$ (and $\tilde{h}_p$ for $\bZ[(1+\sqrt{-p})/2]$), we have the following standard formulas (for $p\neq 3$):
\begin{itemize}
\item The number of supersingular $j$ invariants over $\bF_{p^2}$ is $n=\floor{(p-1)/12} + e_{0} + e_{1728}$, where $e_{x}$ is $0$ or $1$ depending on if $x$ is supersingular at $p$.
\item If $p = 7 \pmod{8}$ then $h_{p} = \tilde{h}_{p}$ and there are $(h_{p}+1)/2$ options for both $O(p,q,r,m)$ and $O'(p,q,r',m')$.
\item If $p = 3 \pmod{8}$ then $h_{p} = 3\tilde{h}_{p}$ and there are $(h_p+1)/2$ options for $O'(p,q,r',m')$ and $(\tilde{h}_p+1)/2$ options for $O'(p,q,r',m')$.
\item If $p= 1 \pmod{4}$ there are $h_p/2$ options for $O(p,q,r,m)$.
\end{itemize}
Combining the above allows us to compute the number of $\bF_p$ rational supersingular values in terms of $h_{p}$.

More generally, if we fix $K=\bQ(\sqrt{-D})$ a quadratic imaginary field of discriminant $-D$ and class number $h_K$.
Fix an order $\cO = \bZ + \ff\cO_K$ and write $\ff = \prod q_i^{a_i}$
The class number of $\cO$ is given by:
\[   h_{\cO} = \epsilon h_K \prod_{i} \left(q_i - \left(\frac{-D}{q_i}\right)\right)q_i^{a_i-1} \]
where $\epsilon = 1$ unless $D=-3$ or $D=-4$.

If $D=-3$ and the formula above is divisible by $3$ then $\epsilon =\tfrac{1}{3}$.
If $D=-4$ and the formula above is divisible by $2$ then $\epsilon =\tfrac{1}{2}$.
\end{rmk}

\begin{thm}[Halter-Koch]
If $n$ is the number of prime divisors of $D\ff$ then:
\[ \abs{C\ell(\cO)[2]} = \begin{cases} 2^{n-1}  & D\ff \text{ odd} \\
                                                              2^{n-2}  & 2 || D\ff  \\
                                                              2^{n-1}  & 4 || D\ff  \\
                                                              2^{n-1}  & 8 || D\ff  \\
                                                              2^{n}  & 16 | D\ff  \end{cases}. \]

More precisely, the ring class field of $\cO$ contains:
\[ \bQ\left(\sqrt{(-1)^{(q-1)/2} q}\right) \]
where $q$ is an odd prime factor of $D\ff$.

If $D = -8m$ then the ring class field of $\cO$ contains:
\[ \bQ\left(\sqrt{(-1)^{(m-1)/2} 2}\right).  \]

If $D = 4\pmod 8$ and $4|\ff$ then the ring class field of $\cO$ contains:
\[ \bQ\left( \sqrt{2}\right) .\]

If $D$ is odd, and $8|\ff$ then the ring class field of $\cO$ contains:
\[ \bQ\left( \sqrt{2}\right). \]

If $D = 4\pmod{8}$, or $2|\ff$ and $2|D$, or $D$ is odd and $4|\ff$ then the ring class field of $\cO$ contains:
\[ \bQ\left(\sqrt{-1}\right). \]

The above fields generate the genus field $F$, moreover, this is thee maximal subextension of the ring class field of $\cO$ generated by quadratic extensions.
\end{thm}
See \cite[Thm 6.1.4]{SchertzCM}.

\section{Results}\label{sec:res}

In this section we will present our main theorems. 
These are primarily structured to address the entries in the data which we present in the Appendix.

We will begin by looking at certain conditions on $\cO$ under which there can be no elliptic curves over $\bZ_p$ which admit CM by $\cO$.
These first results can naturally be viewed as generalizations of those of \cite{MortonGenus} and \cite{StankewiczTwistsofShimura} which can be interpreted as giving the conditions on the odd prime factors of the discriminants.

\begin{thm}\label{thm:nolinear}
Fix $K=\bQ(\sqrt{-D})$ of discriminant $-D$. Fix an order $\cO = \bZ + \ff\cO_K$ of conductor $\ff \in \bZ$ and suppose that $\left(\frac{-D\ff^2}{p}\right) = -1$.
There are no elliptic curves over $\bZ_p$ which admit CM by $\cO$ if any of the following occur:
\begin{itemize}
\item  there is an odd prime factor $q$ of $D\ff$ with $\left(\dfrac{-p}{q}\right) = -1$
\item $p=1\pmod{4}$ and $16 | D\ff^2$.
\item $p=3\pmod{8}$ and $8 | D$.
\item $p=3\pmod{8}$ and $64 | D\ff^2$
\end{itemize}
Otherwise there are exactly $\abs{C\ell(\cO)[2]}$ $j$-invariants for elliptic curves over $\bZ_p$ which admit CM by $\cO$.
\end{thm}

\begin{rmk}
The condition that there is an odd prime factor $q$ of $D$ with $\left(\dfrac{-p}{q}\right) = -1$ implies in particular that the quaternion algebra
  $(-p,-D)$ is ramified at $q$. Though this can be used to justify the condition for those $q|D$, we will not follow this strategy of proof, rather we give a proof which has a more natural connection to class field theory.

The condition on odd primes cannot be extended to even primes by use of the Kronecker symbol, the dependence on the behaviour at $2$ is more subtle.
\end{rmk}

In order to prove the result we shall make use of a few lemmas.
\begin{lemma}
Fix $K=\bQ(\sqrt{-D})$ of discriminant $-D$. Fix an order $\cO = \bZ + \ff\cO_K$ and suppose that $\left(\frac{-D\ff^2}{p}\right) = -1$.
The polynomial $P_{\cO}(X)$ has a linear factor over $\bZ_p$ if and only if $N = \bQ(j(\cO))$ has no quadratic subextension in which $p$ is inert.
\end{lemma}
\begin{proof}
If there is a quadratic subextension of $N$ which is inert at $p$, then all factors of $p$ in $N$ have inertial degree $2$, and thus there can be no linear factors.

Conversely, suppose every factor of $p$ in $N$ has inertial degree $2$.
let $\fp$ be a prime of $L$ over $p$ and let $\sigma$ be a generator for the decomposition group of $\fp$ and  let $\tau$ be a generator of $\Gal(L/N)$.
Then 
\begin{itemize}
\item $\sigma$ is $2$-indivisible ($\sigma\neq x\cdot x$ for any $x$) with exact order $2$, because this is true of $\Frob_p$.
\item $\sigma$ and $\tau$ are not conjugate, since if $\tau$ were a conjugate of $\Frob_p$ the field $N=L^\tau$ would have a non-inert prime.
\item $\sigma$ and $\tau$ commute since $\sigma$ has order $2$.
\item $\sigma\tau$ is in $\Gal(L/K)$ as they both act non-trivially on $K$.
\item It follows from the above, and the basic structure of dihedral groups, that $\sigma\tau$ is indivisible with exact order $2$.
\end{itemize}
Thus we may write:
 \[ \Gal(L/K) = \langle \sigma\tau \rangle \times H \]
and thus
\[  \Gal(L/\bQ)= \langle \sigma \rangle \times (H\rtimes \langle \tau \rangle). \]
We see that $G=(H\rtimes \langle \tau \rangle)$ is a normal subgroup of $\Gal(L/\bQ)$, moreover, the field $L^G$ is an inert quadratic subextension of $N$.
\end{proof}

\begin{lemma}
The maximal subextension of $N$ generated by quadratic extensions is the totally real subfield $M$ of $F$ the genus field of $L$.
\end{lemma}
\begin{proof}
It suffices to show that $N$ has a real embedding since any composite of quadratic extensions is either totally complex or totally real.

To see this we use the fact that:
\[ \overline{j(\fa)} = j(\overline{\fa}).  \]
It is thus sufficient to find $\fa$ such that $\overline{\fa} = \fa$, but indeed we may simply take $\fa = \cO$.
\end{proof}

\begin{proof}[proof of Theorem \ref{thm:nolinear}]
The idea of the proof is to show that $p$ is inert in a quadratic subextension of the totally real subfield $N$ of $F$ if and only if one of the conditions of the theorem holds.

To show this we must find a subextension of $N$ defined by adjoining the square root of a positive integer which is not a square modulo $p$, in each of the following cases we describe how to find such a non-square.
Note that if $q=3\pmod{4}$ then $\sqrt{Dq} \in N$ whereas if $q=1\pmod{4}$ then $\sqrt{q}\in N$.
\begin{itemize}
\item
Consider the case where $p = 1 \pmod{  4}$ and $4 || D$. In this case there exists odd prime factor $q'$ of $D$ with $\left(\dfrac{-p}{q'}\right) = -1$.
Moreover, $D$ has a factor $q$ such that both $\pm q$ are not squares mod $p$.

\item Suppose there is an odd prime factor $q$ of $D\ff$ with $\left(\dfrac{-p}{q}\right) = -1$.
\begin{itemize}
 \item if $q=p=3\pmod 4$ we obtain $\left(\dfrac{q}{p}\right) = -1$ and thus $Dq$ is not a square mod $p$.
 \item if $q=3\pmod 4$, $p=1\pmod{4}$ and $2 \not\vert D$ we obtain $\left(\dfrac{q}{p}\right) = 1$ and thus $Dq$ is not a square mod $p$.
 \item if $q=1\pmod 4$ we obtain $\left(\dfrac{q}{p}\right) = -1$ and thus $q$ and is not a square mod $p$.
\end{itemize}

\item Suppose $p= 3 \pmod{  8} $ and $8|D $ and $D/8 = 3 \pmod 4$ then $D$ has a factor $d$ congruent to $3\pmod 4$ which is not a square mod $p$.
\item Suppose $p= 3 \pmod{  8} $ and $8|D$ and $D/8 = 1 \pmod 4$ then $2$ is not a square mod $p$.

\item Suppose $p=3\pmod{8}$ and $64 | D\ff^2$ then $2$ is not a square mod $p$.

\item Suppose $p=1\pmod{4}$ and $16 | D\ff^2$ then $D$ has a factor $q$ such that both $\pm q$ are not square mod $p$.

\end{itemize}
The above covers all of the cases of the theorem.

To prove the converse we remark that if $p$ is inert in $N$ it is inert in a quadratic subextension of one of the following types: 
\begin{itemize}
\item $\bQ(\sqrt{q})$ where $q | fD$ or 
\item $\bQ(\sqrt{q_1q_2})$ where both $q_1,q_2 = 3\pmod{4}$ and $q_1q_2|fD$.
\end{itemize}
as such fields generate the genus field of $N$.
Completing the proof follows a similar case analysis to the above.
\end{proof}

We now shift to discussing a phenomenon whereby certain $\bF_p$ reductions are disallowed based on the ramification behavior of $2$.
\begin{rmk}
In the following theorem we will be distinguishing the supersingular $j$-invariants in $\bF_p$ by identifying them as roots of $P_{\bZ[\sqrt{-p}]}(X)$ or $P_{\bZ[(1+\sqrt{-p})/2]}(X)$.

To understand the significance we recall the theorems above of Ibukiyama which asserted that this naturally divides the supersingular values into two almost disjoint sets.
More precisely, by \cite{Elkies1987} and \cite{kaneko1989} we have that for $p=3\pmod{4}$ these polynomials factor as $(X-1728)\prod_i (X-\alpha_i)^2$ whereas for $p=1\pmod{4}$ the factorization is $\prod_i (X-\alpha_i)^2$. In each case the $\alpha_i$ are distinct in $\bF_p$.
Furthermore, in the case $p=3\pmod{4}$ the $\alpha_i$ for $P_{\bZ[\sqrt{-p}]}(X)$ are distinct from those for $P_{\bZ[(1+\sqrt{-p})/2]}(X)$.
The polynomial for $\sqrt{-2}$ is precisely $P_{\bZ[\sqrt{-2}]}(X)=X-8000$.
\end{rmk}

\begin{thm}\label{thm:tworam}
Fix $K=\bQ(\sqrt{-D})$ of discriminant $-D$. Fix an order $\cO = \bZ + \ff\cO_K$ of conductor $\ff\in \bZ$ and suppose that $\left(\frac{-D\ff^2}{p}\right) = -1$.
Let $j$ be a $\bZ_p$ root of $P_{\cO}(X)$.
\begin{itemize}
\item Suppose $p=7\pmod{8}$ 
\begin{itemize}
\item If $2$ is unramified in $K$ and $2 \not\vert \ff$ then $j$ is a root of $P_{\bZ[\sqrt{-p}]}(X)$.
\item If $2$ is unramified in $K$ and $2 | \ff$ but $8 \not\vert\ff$ then $j$ is a root of $P_{\bZ[(1+\sqrt{-p})/2]}(X)$.
\item If $2$ is unramified in $K$ and $8 | \ff$ then $j$ is a root of $P_{\bZ[\sqrt{-p}]}(X)$ or $P_{\bZ[(1+\sqrt{-p})/2]}(X)$.
\item If $2$ is tamely ramified in $K$ and $2 \not\vert \ff$ or $4|\ff$ then $j$ is a root of $P_{\bZ[\sqrt{-p}]}(X)$ or $P_{\bZ[(1+\sqrt{-p})/2]}(X)$.
\item If $2$ is tamely ramified in $K$ and $2 || \ff$ then $j$ is a root of $P_{\bZ[(1+\sqrt{-p})/2]}(X)$.
\item If $2$ is wildly ramified in $K$ and $2 \not\vert \ff$ then $j$ is a root of $P_{\bZ[(1+\sqrt{-p})/2]}(X)$.
\item If $2$ is wildly ramified in $K$ and $2 | \ff$ then $j$ is a root of $P_{\bZ[\sqrt{-p}]}(X)$ or $P_{\bZ[(1+\sqrt{-p})/2]}(X)$.
\end{itemize}
\item Suppose $p=3\pmod{8}$ 
\begin{itemize}
\item If $2$ is unramified in $K$ and $2\not\vert\ff$ or $4||\ff$ then $j$ is a root of $P_{\bZ[\sqrt{-p}]}(X)$.
\item If $2$ is unramified in $K$ and $2 || \ff$ then $j$ is a root of $P_{\bZ[(1+\sqrt{-p})/2]}(X)$.
\item If $2$ is unramified in $K$ and $8 | \ff$ then there are no linear terms.
\item If $2$ is tamely ramified in $K$ and $2 \not\vert \ff$ then $j$ is a root of $P_{\bZ[\sqrt{-p}]}(X)$ or $P_{\bZ[(1+\sqrt{-p})/2]}(X)$.
\item If $2$ is tamely ramified in $K$ and $2 || \ff$ then $j$ is a root of $P_{\bZ[\sqrt{-p}]}(X)$.
\item If $2$ is tamely ramified in $K$ and $4 | \ff$ then there are no linear terms.
\item If $2$ is wildly ramified in $K$ then there are no linear terms.
\end{itemize}

\item Suppose $p=1\pmod{4}$
\begin{itemize}
\item If $2$ is unramified in $K$ and $4 \not\vert \ff$  then $j$ is a root of $P_{\bZ[\sqrt{-p}]}(X)$.
\item If $2$ is unramified in $K$ and $4 \vert \ff$ then there are no linear terms.
\item If $2$ is tamely ramified then there are no linear terms.
\item If $2$ is wildly ramified in $K$ and $2 \not\vert \ff$  then $j$ is a root of $P_{\bZ[\sqrt{-p}]}(X)$.
\item If $2$ is wildly ramified in $K$ and $2 \vert \ff$ then there are no linear terms.
\end{itemize}
\end{itemize}
\end{thm}

\begin{nota}
Given that any quaternion algebra $A$ is equipped with a canonical bilinear form, given an element $\alpha\in A$ we shall denote by $\alpha^\perp$ the collection of all elements in $A$ perpendicular to $\alpha$, that is elements $x\in A$ with $\Tr_(\alpha x) = 0$.

Similarly, given a subspace such as $\cO\subset A$, we shall denote $\cO^\perp$, the complementary subspace of $A$ with respect to this pairing.
\end{nota}

To prove this we will make use of the following lemma.
\begin{lemma}\label{lem:frob}
If $E$ is an elliptic curve over $\bZ_p$ which admits CM by $\cO$ which corresponds to a datum $(\cO \subset \bB)$ then the Galois Frobenius $\Frob_p$ acting on $E(\overline{\bQ}_p)$ over $\bZ_p$ induces the endomorphism Frobenius $\widetilde{\Frob_p}$ of $\overline{E}$. 
Moreover we have:
\begin{itemize}
\item $\Frob_p$, the Galois action of Frobenius on $E$, acts on $\cO$ by $x\mapsto \overline{x}$.
\item $\widetilde{\Frob_p}$, the endomorphism of $\overline{E}$,  satisfies $\widetilde{\Frob_p} x = \overline{x} \widetilde{\Frob_p}$ for $x\in \cO$.
\item $\Frob_p^2$, the Galois action of Frobenius on $E$, commutes with $\cO$.
\item $\widetilde{\Frob_p}^2$, the endomorphism of $\overline{E}$, satisfies $\widetilde{\Frob_p}^2 = -p$.
\end{itemize}
In particular $\widetilde{\Frob_p} \in \cO^\perp$ is an element of norm $p$.
\end{lemma}
See \cite{SchertzCM}.

\begin{proof}[proof of Theorem \ref{thm:tworam}]
We must show, using Ibukiyama's classification of maximal orders containing $\sqrt{-p}$, that the only CM-orders in $\alpha^\perp$ are those satisfying the conditions of the theorem.

We note that in selecting the values of $q$, $r$ and $m$ we may assume by replacing $r$ by $r+aq$ that $8|r$.
With this assumption we have that $pq = m \pmod{8}$.
When selecting $q$, $r'$ and $m'$ we must have that $r'$ is odd, when $p=3\pmod 8$ this then implies that $m$ is odd.

We observe the following important facts about $\alpha^\perp$ in the various cases:
\begin{enumerate}
\item For the maximal orders of the form $O'(p,q,r',m')$ we have that $\alpha^\perp$ contains no elements with odd trace.

\item For the maximal orders of the form $O'(p,q,r',m')$ we have that all primitive elements of $\bZ[\sqrt{-p}]^\perp$ are of the form:
 \[ y\beta + z\frac{(r'+\alpha)\beta}{2q} \]
for some choice of $y$ and $z$ coprime.

The square of such an element is:
\[ -y^2q -z^2m - yzr'. \]
Notice that if $p = 3\pmod8$ this cannot be even.
\item For the maximal orders of the form $O(p,q,r,m)$ we have that all primitive elements of odd trace in $\bZ[\sqrt{-p}]^\perp$ are of the form:
 \[ y\beta + z\frac{(r+\alpha)\beta}{q} \]
for some choice of $y$ and $z$ coprime, with $z$ odd.

The square of such an element is:
\[ -y^2q -z^2m - 2yzr \]
modulo $8$ this becomes:
\[ -q(y^2-z^2p). \]
Notice that if this is odd, then $y$ is even and $-q(y^2-z^2p) = pq \pmod{8}$.
Also, if it is even then $y$ and $z$ are both odd and it is divisible by $(1-p) | -q(y^2-z^2p)$.
\end{enumerate}
By considering each of the cases of the theorem, the above allows us to conclude the result.
\end{proof}

\begin{prop}\label{prop:hasroots}
Suppose there exists $\bZ[\sqrt{-D}] = \cO \subset \alpha^\perp$, then $P_\cO(X)$ has $\bZ_p$ roots.
\end{prop}
\begin{proof}
By the above argument we note that $\cO \subset \alpha^\perp$ implies the existence of a solution to:
\[ y^2q + z^2m +2yzr = D \qquad\text{ or }\qquad y^2q +z^2m + yzr' = D.  \]
In the first case, multiplying by $q$ we obtain:
\[qD = y^2q^2 + z^2(p+r^2) + 2yzrq =  z^2p + (yq+rz)^2. \]
reducing modulo $8$ and modulo all the odd prime factors of $D$ the result then follows from Theorem \ref{thm:nolinear}.
In the second case, multiplying by $4q$ we obtain:
\[ 4qD = 4y^2q^2 + z^2(p+r^2) + 2yzrq =  z^2p + (2yq+rz)^2 \]
and the result follows similarly.
\end{proof}

\begin{rmk}
Note that the above does not actually prove the converse to Lemma \ref{lem:frob} though it would provide for an alternate proof for one direction of Theorem \ref{thm:nolinear}.
\end{rmk}

We now explain the phenomenon where in specific circumstances certain $\bF_p$ reductions always occur with the same frequency.
Based on \cite{CornutVatsal2} we should expect that this is caused by systematic collections of isogenies (coming from Hecke relations), and in our case we should expect $2$-isogenies to play a role.
The results here have a similar flavor to those of \cite[pp. 95-96]{BrillMorClassNumbers} where they consider similar questions questions related to the orders $\bZ[\sqrt{-p}]$ and $\bZ[\frac{1+\sqrt{-p}}{2}]$.

\begin{lemma}\label{lem:qchoice}
If $\sqrt{q_2} \in  \alpha^\perp$ then $\cO \simeq O(p,q_2,r,m)$ or $O(p,q_2,r',m')$ for some choice of $r,m$ or $r',m'$.
\end{lemma}
\begin{proof}
By \cite[Prop 2.1 and Rmk 2.2]{Ibukiyama} the conditions:
\[q_1q_2 =  z^2p + (yq_1+rz)^2 \qquad \text{or} \qquad 4q_1q_2  = z^2p + (2yq_1+rz)^2 \]
imply that $q_1$ and $q_2$ satisfy 
\[ O(p,q_1,r_1,m_1) \simeq O(p,q_2,r_2,m_2) \text{ or respectively }  O'(p,q_1,r_1',m_1') \simeq O'(p,q_2,r_2',m_2').\]
 The results then follow from the proof of Proposition \ref{prop:hasroots}.
\end{proof}

\begin{lemma}\label{lem:3mod4cong}
Fix $p=3\pmod{4}$.
Fix $K=\bQ(\sqrt{-D})$ of discriminant $-D$. Fix an order $\cO = \bZ + \ff\cO_K$ and suppose that $\left(\frac{-D\ff^2}{p}\right) = -1$.
Suppose further that $2$ is tamely ramified in $K$ but $2$ does not divide $\ff$.

Suppose that $\cO$ is optimally embedded in $O(p,q,r,m)$ and contained in $\alpha^\perp$.
Let $\fa^2 = (2)$ in $\cO$.
Then $\fa O(p,q,r,m) \fa^{-1} \simeq O'(p,q,r',m')$ is a maximal order with an optimal embedding of $\cO$.
Consequently, if $E$ is an elliptic curve over $\bZ_p$ which admits CM by $\cO$  whose reduction has endomorphism ring $O(p,q,r,m)$, then the reduction of $\fa\ast E$ has endomorphism ring $O'(p,q,r',m')$ with the exact same choice of $q$.

Conversely, if $E$ is an elliptic curve over $\bZ_p$ which admits CM by $\cO$ whose reduction has endomorphism ring $O'(p,q,r',m')$, then the reduction of $\fa\ast E$ has endomorphism ring $O(p,\tilde{q},\tilde{r},\tilde{m})$ for some $\tilde{q}$ such that $O'(p,q,r',m')\simeq O'(p,\tilde{q},\tilde{r}',\tilde{m}')$.
\end{lemma}
\begin{proof}
Let $\cO = \bZ[\gamma = \sqrt{q_2}]$.
It suffices to show that $\fa O(p,q,r,m) \fa^{-1}$ contains both $\frac{1+\alpha}{2}$ and $\beta$.

We note that $\fa = (2,1+\gamma)$ and $\fa^{-1} = (1,\frac{1-\gamma}{2})$. It follows immediately that $\beta \in \fa O(p,q_1,r,m) \fa^{-1}$.

Now we may write $\gamma = y\beta + z\frac{r+\alpha}{q}\beta$ with $y$ and $r$ even and $z$ odd.
Now, by observing that we may write $ \left(\frac{1+\alpha}{2}\right)$ as:
\[ (1+\gamma)\left(\tfrac{1}{2}(-zm+ry+1) + (zm+ry)\left(\frac{1+\beta}{2}\right) - \tfrac{1}{2}(yq+zr)\left(\frac{r+\alpha}{q}\beta\right)\right)\left(\frac{1-\gamma}{2}\right) \]
and that this quantity is an element of $\fa O(p,q,r,m) \fa^{-1}$
we conclude by Lemma \ref{lem:qchoice} that 
\[ \fa O(p,q,r,m) \fa^{-1} \simeq O'(p,q,r',m'). \]

Now suppose we start with $\cO$ optimal in $O'(p,q,r',m')$.
Attempting to reverse the above calculation cannot work in general as we no longer have $r$ and $m$ but $r'$ and $m'$. However, we observe that:
 \[\left((1+\gamma)\left(\frac{1+\alpha}{2}\right)\left(\frac{1-\gamma}{2}\right) - \left(\frac{1+\gamma^2}{4}\right)\alpha\right) \in \fa O'(p,q,r',m') \fa^{-1}\]
is perpendicular to $\alpha$ and has odd trace.
Hence, $\fa O'(p,q,r',m') \fa^{-1} \simeq O(p,\tilde{q},\tilde{r},\tilde{m})$. The result now follows.
\end{proof}

\begin{rmk}
Note, that we could not simply run the first part of the above argument in the opposite direction to go from $O'(p,q,r',m')$ to $O(p,q,r,m)$, in particular this would be impossible in any case where the class groups which classify $O'(p,q,r',m')$ and $O(p,q,r,m)$ are not in bijection. 
\end{rmk}

\begin{thm}\label{thm:7freq}
Fix $p =3\pmod 4$.
Fix $K=\bQ(\sqrt{-D})$ of discriminant $-D$. Fix an order $\cO = \bZ + \ff\cO_K$ and suppose that $\left(\frac{-D\ff^2}{p}\right) = -1$.
Suppose further that $2$ is tamely ramified in $K$ but $2$ does not divide $\ff$.

It we consider the set of supersingular values of $\bF_p$ except $1728$, each $j$-invariant $J$ has a partner $\tilde{J}$ such that, the frequency of the appearance of $X-J$ and $X-\tilde{J}$ as the reduction of irreducible linear factors of $P_{\cO}(X)$ modulo $p$ is the same.
\end{thm}
\begin{proof}
We first observe that if $E$ is defined over $\bZ_p$ then so too is $\fa\ast E$. This follows by observing that the collection of endomorphisms in $\fa$ is Galois stable.
Moreover, in the case $p = 3\pmod 4$ the map from $O(p,q,r,m)$ to $O'(p,q,r',m')$ being injective implies it is bijective as the collections have the same size.

By Lemma \ref{lem:3mod4cong} it now follows that $O(p,q,r,m)$ and $O'(p,q,r',m')$ must occur with the same frequency.

We note that $j$-invariant $1728$ is the only one that can ever be identified with itself through this process, and in fact it must, because the class group has odd order.
\end{proof}

\begin{rmk}
For $p=3\pmod 4$ we obtain other less obvious relationships between the counts for maximal orders of type $O'$ and of type $O$ arising from the fact that the map is generically $3:1$. In particular, in general the frequency for those of type $O'$ is the sum of the frequencies of a specific collection of three of orders of type $O$.
We note that there will be a curve which is $2$-isogenous to the one with $j$-invariant $1728$.

We should point out that the $\bF_p$ points of the $2$-torsion is well understood, that there is a unique $\bF_p$ rational $2$-torsion point is suggestive of the above results, but does not show that the association is between $O(p,q,r,m)$ and $O'(p,q,r',m')$ and certainly not that it `respects $q$'.
\end{rmk}

\begin{thm}\label{thm:1freq}
Fix $p=1\pmod 4$.
Fix $K=\bQ(\sqrt{-D})$ of discriminant $-D$. Fix an order $\cO = \bZ + \ff\cO_K$ and suppose that $\left(\frac{-D\ff^2}{p}\right) = -1$.
Suppose further that $2$ is wildly ramified in $K$ but $2$ does not divide $\ff$.

It we consider the set of supersingular values of $\bF_p$, each $j$-invariant $J$ has a partner $\tilde{J}$ such that, the frequency of the appearance of $X-J$ and $X-\tilde{J}$ as the reduction of irreducible linear factors of $P_{\cO}(X)$ modulo $p$ is the same.

This partner $\tilde{J}$ is independent of $K$ and $\cO$ and depends only on $p$.
\end{thm}
\begin{proof}
Set $\fa^2 = (2)$ in $\cO$. 
In this case we have $\fa = (2,\gamma)$ and $\fa^{-1} = (1,\tfrac{1}{2}\overline{\gamma})$.
As in the previous case, we must only show that $\fa O(p,q,r,m) \fa^{-1}$ is independent of $\cO$.

Now set $\fb^2 = (2)$ in $\bZ[\sqrt{-p}]$.
We have that $\fb = (2,1+\alpha)$.

We recall that we have $\gamma = y\beta + z\frac{r+\alpha}{q}\beta = \tfrac{1}{q}(yq +zr + z\alpha)\beta$ with $r$ even and both $y$ and $z$ odd.

We claim that $(1+\alpha) \in \fa O(p,q,r,m)$. Indeed, as $\beta\in O(p,q,r,m)$ we have $yq +zr + z\alpha = \gamma\beta \in \fa O(p,q,r,m)$. Since $2\in \fa$ the claim then follows immediately.
Conversely, it is clear that $q\gamma \in \fb O(p,q,r,m)$. As $q$ is odd, and $2 \in \fb$ we also have that $\gamma\in \fb O(p,q,r,m)$.
We thus have shown that $\fa O(p,q,r,m) = \fb O(p,q,r,m)$.

It follows that $\fa O(p,q,r,m) \fa^{-1} = \fb O(p,q,r,m) \fb^{-1}$ is independent of $\cO$.
\end{proof}

\begin{rmk}
In this case the uniqueness of the $\bF_p$-rational $2$-torsion points is sufficient to conclude the result.
\end{rmk}

\section{Further Questions}\label{sec:cq}

Our results suggest the following natural questions:

\begin{qu}
In Theorem \ref{thm:tworam} we gave necessary conditions for a datum $(\cO\subset\bB)$ to correspond to an elliptic curve over $\bZ_p$. Moreover, Proposition \ref{prop:hasroots} gives the impression that this may be sufficient.
It is natural to ask, if these conditions are in fact sufficient.

\begin{enumerate}
\item[(a)]
More precisely, given an elliptic curve over $\bF_p$, and an endomorphism (defined over some extension) when can we lift the curve to $\bZ_p$ such that the endomorphism lifts to some extension?
\item[(b)]
Is it sufficient that the endomorphism be perpendicular to Frobenious in the endomorphism algebra over $\overline{\bF_p}$? 
\end{enumerate}
\end{qu}
An answer to this question would shed light on the structure of the ring  $\bZ[j(E_1),\ldots,j(E_n)]$ as remarked in the introduction.

\begin{qu}
Theorems \ref{thm:7freq} and \ref{thm:1freq} give situations in which there are automatic relationships between certain roots of $P_\cO(X)$.
As remarked a similar result holds for the same reason when $p=3\pmod 8$.
\begin{enumerate}
\item[(a)]
It is natural to ask if there are other situations in such relationships must exist? In particular are there situations where the role of $2$ can be replaced by some other prime?
\item[(b)]
The method of proof also suggests that we could anticipate relations between the roots of $P_\cO(X)$ between two different orders in the same field whose conductors differ by a factor of $2$. Can the combinatorics of this be made more precise?
\end{enumerate}
\end{qu}

\section*{Acknowledgements}

 I would like to thank Prof. Eyal Goren for suggesting the computations which led to the discovery of these results.
 I would like to thank Prof. Ernst Kani for some useful discussions as well as recommending several references.
 I would like to thank the referee, for their numerous helpful suggestions.
I would like to thank the many developers of SAGE without which the computations through which we uncovered these results would not be possible.
I would also like to thank the SAGE Notebook project for the use of various computing resources which they have made available through their various funding sources.

{}\ifx\XMetaCompile\undefined

\providecommand{\MR}[1]{}
\providecommand{\bysame}{\leavevmode\hbox to3em{\hrulefill}\thinspace}
\providecommand{\MR}{\relax\ifhmode\unskip\space\fi MR }
\providecommand{\MRhref}[2]{  \href{http://www.ams.org/mathscinet-getitem?mr=#1}{#2}
}
\providecommand{\href}[2]{#2}

{}\fi

\appendix
\section{Data}\label{sec:ad}

In this section we will present a representative sample of the data, which forms the basis for how we discovered the theorems. Similar computations have been done for all $p$ up to $1000$. All of these computation were performed in SAGE. Data not contained can be obtained from the author.

In all the data which follows, the frequencies presented represent the total number of times each factor appears as the reduction modulo $p$ of an irreducible factor of $P_{\cO}(X)$ over $\bZ_p$ for one of the orders under consideration.
We will consider several families of orders, but in all cases we are considering all orders in the described class with class numbers strictly less than $40$ (and discriminant of the base field less than $10$ million, noting that there are no fields of class number less than $100$ with discriminant between $3$ and $10$ million).

Made precise, the appearance of the $199$ in the first table indicates that there are exactly $199$ different $j$-invariants congruent to $0$ mod $71$ for elliptic curves over $\bZ_{71}$ which admit CM by the maximal order of a quadratic imaginary field with odd class number less than $40$ (and discriminant less than $10$ million) in which $71$ is inert.
The $1$ in the first table indicates that there is a unique $j$-invariant over $\bZ_{71}$ congruent to $-23$ modulo $71$ for which the associated elliptic curve admits CM by the maximal order of a quadratic imaginary field with odd class number less than $40$ (and discriminant less than $10$ million) in which $71$ is inert.
We remark that this unique $j$-invariant is that of the curve which admits CM by $\bZ[\sqrt{-2}]$.

We should remark that ordering by class number is not ideal, in particular this appears to change the relative frequency of various congruence conditions. Heuristically this can be explained for example by noticing that if $2$ ramifies, this will tend to double the size of the class group, whereas if $2$ splits this will tend to make it much larger. That is the splitting and ramification of small primes tends to impact the size of the class group as it is precisely these primes which, by Minkowski theory, will be the generators.
Moreover, adding factors of $2$ to the conductor will typically double (or triple) the class number.
The effect is that in the data which follows you should not try to compare data between columns without adjusting for the bias caused by the class number cutoffs.

We should note, it is entirely possible that the skew the class number ordering creates in the data is the only reason we were able to originally identify any of the underlying phenomenon we have discussed. In particular, considering parity conditions on the class number is likely entirely unnatural.

\subsection{Data for $p=71$}
This data is typical for $p = 7\pmod8$, the choices of $5$ and $7$ are arbitrary but demonstrate contrasting behaviour.

\begin{center}
All Orders Inert at $p=71$ subject to conditions on discriminants/conductors.
\setlength\tabcolsep{1.5pt}.
{\tiny
\begin{tabular}{ | c | c |c | c |c | c |c | c |c | c |c | c |c |c | c |c | c |c | c |c | c |c | c |  }
\hline
		&	     All	& $2\not\vert D$	&	  $2\not\vert D$	&	  $2\not\vert D$	&	   	   $2\not\vert D$	&	    $2\not\vert D$& $4\vert\vert D$&	  $4\vert\vert D$&	  $4\vert\vert D$&	  	   $4\vert\vert D$	&	    $4\vert\vert D$	&	     $8\vert\vert D$&	  $8\vert\vert D$&	   $8\vert\vert D$	&	   $8\vert\vert D$	&	    $8\vert\vert D$	&	   $7|D\ff$    &   $7\not\vert D\ff$	\\
		&	     All	& $2\not\vert \ff$	&	  $2\vert\vert \ff$	&	 $4\vert\vert \ff$	&	   $8\vert\vert \ff$	&	   $16\vert \ff$	& $2\not\vert \ff$	&	  $2\vert\vert \ff$	&	  $4\vert\vert \ff$	&	  $8\vert\vert \ff$	&	   $16\vert \ff$		&	   $2\not\vert \ff$	&	$2\vert\vert \ff$	&	  $4\vert\vert \ff$	&	   $8\vert\vert \ff$	&	  $16\vert \ff$		&	      &   	\\
\hline
$	               x	         $&	  1109	&	  806	&	   -	&	   -	&	  18	&5	  	&	  158	&	    -	&	   17	&	     4	&3	     	&	   -	&	    73	&	    16	&	     5	&4	    	&	   -	&	 1109	\\
$	           x + 5	         $&	  1123	&	  817	&	   -	&	   -	&	  20	&5	  	&	  152	&	    -	&	   16	&	     6	&3	    	&	   -	&	    74	&	    22	&	     6	&2	     	&	   -	&	 1123	\\
$	          x + 23	         $&	   941	&	    -	&	 173	&	  82	&	  14	&5	  	&	  152	&	   75	&	   17	&	     4	&2	     	&	 314	&	    73	&	    21	&	     6	&3	     	&	   -	&	  941	\\
$	          x + 30         	$&	  1126	&	  811	&	   -	&	   -	&	  22	&8	 	&	  161	&	    -	&	   11	&	     6	&3	     	&	   -	&	    74	&	    23	&	     4	&3	     	&	   -	&	 1126	\\
$	          x + 31	         $&	   967	&	    -	&	 176	&	  94	&	  26	&6	   	&	  158	&	   77	&	   11	&	     6	&3	     	&	 303	&	    75	&	    23	&	     6	&3	     	&	   -	&	  967	\\
$	          x + 47	         $&	  1027	&	  408	&	  86	&	  48	&	  14	&6	   	&	  143	&	   39	&	   17	&	     6	&4	   	&	 155	&	    74	&	    22	&	     3	&2	     	&	   -	&	 1027	\\
$	          x + 54	         $&	   934	&	    -	&	 169	&	  86	&	  18	&5	   	&	  161	&	   66	&	   17	&	     6	&4	     	&	 301	&	    79	&	    15	&	     4	&3	     	&	   -	&	  934	\\
$	             x^2	         $&	  2981	&	 1572	&	 432	&	 101	&	  19	&3	   	&	  298	&	   90	&	   12	&	     4	&4	     	&	 378	&	    47	&	    19	&	     -	&2	     	&	 467	&	 2514	\\
$	       (x + 5)^2	$&	 10258	&	 5675	&	1293	&	 310	&	  70	&11	   	&	 1078	&	  267	&	   55	&	    16	&16	   	&	1164	&	   207	&	    72	&	     8	&16	   	&	1447	&	 8811	\\
$	      (x + 23)^2	$&	 10375	&	 6106	&	1194	&	 278	&	  65	&9	   	&	 1086	&	  229	&	   60	&	    15	&14	   	&	1009	&	   219	&	    63	&	    11	&17	   	&	1427	&	 8948	\\
$	      (x + 30)^2	$&	 10214	&	 5661	&	1292	&	 304	&	  56	&11	   	&	 1068	&	  271	&	   60	&	    15	&12	    	&	1159	&	   208	&	    69	&	    14	&14	  	&	1418	&	 8796	\\
$	      (x + 31)^2	$&	 10283	&	 6052	&	1213	&	 251	&	  62	&16	  	&	 1062	&	  236	&	   61	&	    20	&12	    	&	1001	&	   207	&	    73	&	     7	&10	    	&	1414	&	 8869	\\
$	      (x + 47)^2	$&	  4833	          &	 2787	&	 598	&	 134	&	  28	&4	 	&	  499	&	  118	&	   26	&	    11	&6	   	&	 509	&	    78	&	    28	&	     -	&7	    	&	 721	&	 4112	\\
$	      (x + 54)^2	$&	 10255	&	 6042	&	1187	&	 258	&	  56	&14	  	&	 1065	&	  235	&	   54	&	    14	&17	    	&	1001	&	   218	&	    71	&	    11	&12	     	&	1450	&	 8805	\\
\hline
\end{tabular}
}
\end{center}

\subsection{Data for $p=59$}
This data is typical for $p = 3\pmod8$.
\begin{center}
All Orders Inert at $p=59$ subject to conditions on discriminants/conductors.

\setlength\tabcolsep{2.5pt}.
{\tiny
\begin{tabular}{ | c | c |c | c |c | c |c | c |c | c |c | c |c |c | c |c | c |c | c |c | c |c | c |  }
\hline
		&	     All	& $2\not\vert D$	&	  $2\not\vert D$	&	  $2\not\vert D$	&	   	   $2\not\vert D$	&	    $2\not\vert D$& $4\vert\vert D$&	  $4\vert\vert D$&	  $4\vert\vert D$&	  	   $4\vert\vert D$	&	    $4\vert\vert D$	&	     $8\vert\vert D$&	  $8\vert\vert D$&	   $8\vert\vert D$	&	   $8\vert\vert D$	&	    $8\vert\vert D$		\\
		&	     All	& $2\not\vert \ff$	&	  $2\vert\vert \ff$	&	 $4\vert\vert \ff$	&	   $8\vert\vert \ff$	&	   $16\vert \ff$	& $2\not\vert \ff$	&	  $2\vert\vert \ff$	&	  $4\vert\vert \ff$	&	  $8\vert\vert \ff$	&	   $16\vert \ff$		&	   $2\not\vert \ff$	&	$2\vert\vert \ff$	&	  $4\vert\vert \ff$	&	   $8\vert\vert \ff$	&	  $16\vert \ff$	 	\\
\hline
$	               x	$&	  1245	&	  896	&	   -	&	  92	&	   -	&	   -	&	  172	&	   85	&	    -	&	     -	&	     -	&	   -	&	     -	&	     -	&	-	&	     -\\
$	          x + 11	$&	  1241	&	  890	&	   -	&	  98	&	   -	&	   -	&	  173	&	   80	&	    -	&	     -	&	     -	&	   -	&	     -	&	     -	&	-	&	     -\\
$	          x + 12	$&	  1236	&	  890	&	   -	&	  97	&	   -	&	   -	&	  167	&	   82	&	    -	&	     -	&	     -	&	   -	&	     -	&	     -	&	-	&	     -\\
$	          x + 31	$&	  1224	&	  870	&	   -	&	  91	&	   -	&	   -	&	  172	&	   91	&	    -	&	     -	&	     -	&	   -	&	     -	&	     -	&	-	&	     -\\
$	          x + 42	$&	  1146	&	  440	&	 285	&	  40	&	   -	&	   -	&	  336	&	   45	&	    -	&	     -	&	     -	&	   -	&	     -	&	     -	&	-	&	     -\\
$	          x + 44	$&	  1060	&	    -	&	 549	&	   -	&	   -	&	   -	&	  511	&	    -	&	    -	&	     -	&	     -	&	   -	&	     -	&	     -	&	-	&	     -\\
$	      (x + 11)^2	$&	 12375	&	 6855	&	1574	&	 325	&	 103	&	29	&	 1269	&	  299	&	   85	&	    28	&	16	&	1389	&	   297	&	    92	&	    12	&	     2\\
$	      (x + 12)^2	$&	 12241	&	 6818	&	1537	&	 319	&	  98	&	27	&	 1229	&	  293	&	   84	&	    35	&	18	&	1371	&	   306	&	    93	&	    10	&	     3\\
$	      (x + 31)^2	$&	 12274	&	 6844	&	1544	&	 324	&	  90	&	27	&	 1250	&	  282	&	   83	&	    43	&	14	&	1381	&	   292	&	    83	&	    12	&	     5\\
$	      (x + 42)^2	$&	  5910	&	 3429	&	 632	&	 160	&	  50	&	13	&	  511	&	  144	&	   39	&	    18	&	8	&	 701	&	   145	&	    52	&	     6	&	     2\\
$	      (x + 44)^2	$&	 12360	&	 7250	&	1264	&	 381	&	 114	&	29	&	 1066	&	  329	&	   80	&	    32	&	14	&	1399	&	   300	&	    87	&	    12	&	     3\\
$	             x^2	$&	  3692	&	 1983	&	 512	&	  77	&	  31	&	7	&	  352	&	   72	&	   26	&	    14	&	6	&	 474	&	   100	&	    33	&	     4	&	     1\\
\hline
\end{tabular}
}
\end{center}

\subsection{Data for $p=41$}
This data is typical for $p = 1\pmod4$.
\begin{center}
All Orders Inert at $p=41$ subject to conditions on discriminants/conductors.
\setlength\tabcolsep{2.5pt}.
{\tiny
\begin{tabular}{ | c | c |c | c |c | c |c | c |c | c |c | c |c |c | c |c | c |c | c |c | c |c | c |  }
\hline
		&	     All	& $2\not\vert D$	&	  $2\not\vert D$	&	  $2\not\vert D$	&	   	   $2\not\vert D$	&	    $2\not\vert D$& $4\vert\vert D$&	  $4\vert\vert D$&	  $4\vert\vert D$&	  	   $4\vert\vert D$	&	    $4\vert\vert D$	&	     $8\vert\vert D$&	  $8\vert\vert D$&	   $8\vert\vert D$	&	   $8\vert\vert D$	&	    $8\vert\vert D$		\\
		&	     All	& $2\not\vert \ff$	&	  $2\vert\vert \ff$	&	 $4\vert\vert \ff$	&	   $8\vert\vert \ff$	&	   $16\vert \ff$	& $2\not\vert \ff$	&	  $2\vert\vert \ff$	&	  $4\vert\vert \ff$	&	  $8\vert\vert \ff$	&	   $16\vert \ff$		&	   $2\not\vert \ff$	&	$2\vert\vert \ff$	&	  $4\vert\vert \ff$	&	   $8\vert\vert \ff$	&	  $16\vert \ff$	 	\\
\hline
$	               x	$&	  1488	&	 1055	&	 222	&	   -	&	   -	&	   -	&	    -	&	    -	&	    -	&	     -	&	     -	&	 211	&	     -	&	     -	&	     -	&	     -	\\
$	           x + 9	$&	  1495	&	 1068	&	 220	&	   -	&	   -	&	   -	&	    -	&	    -	&	    -	&	     -	&	     -	&	 207	&	     -	&	     -	&	     -	&	     -	\\
$	          x + 13	$&	  1491	&	 1055	&	 229	&	   -	&	   -	&	   -	&	    -	&	    -	&	    -	&	     -	&	     -	&	 207	&	     -	&	     -	&	     -	&	     -	\\
$	          x + 38	$&	  1499	&	 1065	&	 223	&	   -	&	   -	&	   -	&	    -	&	    -	&	    -	&	     -	&	     -	&	 211	&	     -	&	     -	&	     -	&	     -	\\
$	             x^2	$&	  5583	&	 3036	&	 665	&	184	&	  46	&	13	&	  675	&	  146	&	   37	&	     8	&	     2	&	 560	&	   146	&	    45	&	    13	&	     7	\\
$	       (x + 9)^2	$&	 18184	&	10102	&	2215	&	557	&	 143	&	48	&	 2014	&	  454	&	  117	&	    37	&	     3	&	1877	&	   434	&	   135	&	    33	&	    15	\\
$	      (x + 13)^2	$&	 18218	&	10107	&	2199	&	582	&	 133	&	52	&	 2001	&	  444	&	  123	&	    37	&	     3	&	1906	&	   443	&	   132	&	    43	&	    13	\\
$	      (x + 38)^2	$&	 18173	&	10080	&	2205	&	583	&	 138	&	47	&	 2015	&	  432	&	  131	&	    38	&	     8	&	1871	&	   437	&	   128	&	    47	&	    13	\\

\hline
\end{tabular}
}
\end{center}

\subsection{Key Observations About Data}

\begin{itemize}
\item In all of the data sets the frequency with which the roots appear appears to be equidistributed subject only to rescaling those $j$ invariants for which the curves have automorphisms.

\item The linear terms do not follow the same distribution as the underlying roots.

\item It is not immediately clear if we restrict to maximal orders if the linear terms are equidistributed overall, however each family based on ramification at $2$ appears to be, and if we reweigh (to correct bias caused by class number bounds) and regroup it is possible that the result is equidistribution.

\item Specifying ramification conditions can have an effect on the presence or absence of linear factors.

\item All families where we account for the behaviour at $2$ in the discriminant and conductor appears to satisfy a simple distribution.
\begin{itemize}
\item The frequency of $j$-invariants with automorphisms may or may not be effected depending on the case.

More specifically $j=0$ is never apparently rescaled, whereas $j=1728$ may or may not be depending on discriminant and conductor.
\item Some $j$-invariants may be favoured despite no extra automorphisms.

For example, $j=-44$ for $p=59$ when $4||D$ and $2$ does not divide $\ff$.

\item For $p=3\pmod4$ there is a partitioning of $j$-invariants into two sets (with $j=1728$ the common intersection) where the distribution selects for one set or the other based on the conditions on discriminants and conductors.

\item For $p=7\pmod 8$, $4||D$ and $2\not\vert\ff$ there is an apparent bijection between these two sets (excluding $1728$) where the frequencies will be identical between the two sets.

Note: Within the data this actually happens on the level of individual orders.

\item For $p=1\pmod 4$, $8||D$ and $2\not\vert\ff$ there is an apparent bijection between two sets where the frequencies will be identical between the two sets.

Note: Within the data this actually happens on the level of individual orders.
\end{itemize}

\item Based on heuristic reasoning on the effect on class number of changing conductors by $2$ and the apparent patterns and equidistribution in families one can reasonably expect equidistribution of the linear terms in the limit if we consider all conductors in a given field.

That is, we know the relative sizes of the exceptional sets and the effect on class numbers of increasing conductors by factors of $2$.

\end{itemize}

{}\ifx\XMetaCompile\undefined


\begin{thebibliography}{{Mor}14}

\bibitem[BM04]{BrillMorClassNumbers}
John Brillhart and Patrick Morton, \emph{Class numbers of quadratic fields,
  hasse invariants of elliptic curves, and the supersingular polynomial},
  Journal of Number Theory \textbf{106} (2004), no.~1, 79 -- 111.

\bibitem[CV05]{CornutVatsal1}
C.~Cornut and V.~Vatsal, \emph{C{M} points and quaternion algebras}, Doc. Math.
  \textbf{10} (2005), 263--309. \MR{2148077 (2006c:11069)}

\bibitem[CV07]{CornutVatsal2}
Christophe Cornut and Vinayak Vatsal, \emph{Nontriviality of {R}ankin-{S}elberg
  {$L$}-functions and {CM} points}, {$L$}-functions and {G}alois
  representations, London Math. Soc. Lecture Note Ser., vol. 320, Cambridge
  Univ. Press, Cambridge, 2007, pp.~121--186. \MR{2392354 (2009m:11088)}

\bibitem[Deu41]{Deuring}
Max Deuring, \emph{Die {T}ypen der {M}ultiplikatorenringe elliptischer
  {F}unktionenk\"orper}, Abh. Math. Sem. Univ. Hamburg \textbf{14} (1941),
  no.~1, 197--272. \MR{3069722}

\bibitem[Dor89]{Dorman1989global}
David~R Dorman, \emph{Global orders in definite quaternion algebras as
  endomorphism rings for reduced cm elliptic curves}, Th{\'e}orie des nombres
  (Quebec, PQ, 1987) (1989), 108--116.

\bibitem[Elk87]{Elkies1987}
Noam~D. Elkies, \emph{The existence of infinitely many supersingular primes for
  every elliptic curve over ℚ}, Inventiones mathematicae \textbf{89} (1987),
  no.~3, 561--567.

\bibitem[Ibu82]{Ibukiyama}
Tomoyoshi Ibukiyama, \emph{On maximal orders of division quaternion algebras
  over the rational number field with certain optimal embeddings}, Nagoya Math.
  J. \textbf{88} (1982), 181--195. \MR{683249 (85c:11112)}

\bibitem[JK11]{JetchevKane}
Dimitar Jetchev and Ben Kane, \emph{Equidistribution of {H}eegner points and
  ternary quadratic forms}, Math. Ann. \textbf{350} (2011), no.~3, 501--532.
  \MR{2805634 (2012k:11081)}

\bibitem[Kan89]{kaneko1989}
Masanobu Kaneko, \emph{Supersingular $j$-invariants as singular moduli ${\rm
  mod}\, p$}, Osaka J. Math. \textbf{26} (1989), no.~4, 849--855.

\bibitem[LV15]{lauter2015arithmetic}
Kristin Lauter and Bianca Viray, \emph{An arithmetic intersection formula for
  denominators of igusa class polynomials}, American Journal of Mathematics
  \textbf{137} (2015), no.~2, 497--533.

\bibitem[{Mor}14]{MortonGenus}
P.~{Morton}, \emph{{Genus theory and the factorization of class equations over
  $\mathbb{F}\_p$}}, ArXiv e-prints (2014).

\bibitem[Sch10]{SchertzCM}
Reinhard Schertz, \emph{Complex multiplication}, New Mathematical Monographs,
  vol.~15, Cambridge University Press, Cambridge, 2010. \MR{2641876
  (2011i:11090)}

\bibitem[{Sta}12]{StankewiczTwistsofShimura}
J.~{Stankewicz}, \emph{{Twists of Shimura Curves}}, ArXiv e-prints (2012).

\end{thebibliography}
\end{document}
{}\fi